\documentclass[11pt,reqno]{amsart}

\usepackage{amssymb,mathtools}
\usepackage{microtype}
\usepackage{url}
\usepackage{aliascnt}
\usepackage[hidelinks]{hyperref}
\usepackage[capitalize,nameinlink]{cleveref}
\crefname{equation}{equation}{equations}
\Crefname{equation}{Equation}{Equations}
\crefname{section}{Section}{Sections}
\crefname{subsection}{Section}{Sections}

\theoremstyle{plain}
\newtheorem{theorem}{Theorem}[section]
\newaliascnt{proposition}{theorem}
\newtheorem{proposition}[proposition]{Proposition}
\aliascntresetthe{proposition}
\crefname{proposition}{Proposition}{Propositions}
\newaliascnt{lemma}{theorem}
\newtheorem{lemma}[lemma]{Lemma}
\aliascntresetthe{lemma}
\crefname{lemma}{Lemma}{Lemmas}
\newaliascnt{corollary}{theorem}
\newtheorem{corollary}[corollary]{Corollary}
\aliascntresetthe{corollary}
\crefname{corollary}{Corollary}{Corollaries}
\theoremstyle{definition}
\newaliascnt{remark}{theorem}
\newtheorem{remark}[remark]{Remark}
\aliascntresetthe{remark}
\crefname{remark}{Remark}{Remarks}
\numberwithin{equation}{section}

\newcommand{\E}{\mathbb{E}}
\newcommand{\Pp}{\mathbb{P}}
\newcommand{\R}{\mathbb{R}}
\newcommand{\C}{\mathbb{C}}
\newcommand{\1}{\mathbf{1}}
\newcommand{\Tr}{\operatorname{Tr}}
\newcommand{\Var}{\operatorname{Var}}
\newcommand{\ESD}{\operatorname{ESD}}
\newcommand{\op}{\mathrm{op}}
\newcommand{\F}{\mathrm{F}}
\newcommand{\dd}{\,\mathrm{d}}
\newcommand{\eps}{\varepsilon}
\newcommand{\Imc}{\operatorname{Im}}
\newcommand{\supp}{\operatorname{supp}}
\newcommand{\law}{\operatorname{Law}}
\newcommand{\NC}{\mathrm{NC}}

\title[Critical tensor covariance at the MP threshold]{Critical tensor
covariance at the Marchenko--Pastur threshold}
\author{Xiaohui Xie}
\address{University of California,
Irvine, California 92697, USA}
\email{xhx@uci.edu}
\date{September 2026}

\begin{document}

\begin{abstract}
  For a centered, variance-one random variable $X$ with finite fourth moment,
  let $x$ be the vector of square-free degree-$d$ monomials in $n$ independent
  copies of $X$. At the critical scale $d^2/n\to\lambda\in(0,\infty)$, the
  normalized squared length of $x$ converges to the lognormal variable
  $R=\exp(\sqrt{\lambda v}\,Z-\lambda v/2)$, where $v=\E X^4-1$ and $Z$ is
  standard normal. If $\binom nd/N\to c\in(0,\infty)$, the sample covariance of
  $N$ independent copies of $x$ has an almost-sure limiting spectral law: the
  free compound-Poisson law with rate $1/c$ and jump distribution $\law(cR)$. It
  reduces to Marchenko--Pastur when $\lambda v=0$. The proof shows that
  subtracting the contribution of the sample length leaves vanishing
  quadratic-form fluctuations, even when $\E X^3\ne0$; length and direction need
  not be independent.
\end{abstract}

\maketitle

\section{Introduction}

Let $X_1,\ldots,X_n$ be independent copies of a real random variable $X$ with
$\E X=0$, $\E X^2=1$ and $\E X^4<\infty$. The principal, or square-free, tensor
of degree $d$ is
\[
  x=(x_I)_{I\in\binom{[n]}d},
  \qquad x_I=\prod_{i\in I}X_i,
  \qquad p=\binom nd.
\]
Its coordinates are uncorrelated with unit variance, but share many of the same
factors. For $N$ independent copies, consider
\[
  S_n=\frac1N\sum_{\alpha=1}^N x^{(\alpha)}x^{(\alpha)\mathsf T},
  \qquad \frac pN\longrightarrow c\in(0,\infty).
\]
For $d=1$, the empirical spectral distribution (ESD) converges to the
Marchenko--Pastur (MP) law~\cite{MarchenkoPastur1967}. How far can $d$ grow
before the shared factors change the limit?

Bryson, Vershynin and Zhao~\cite{BrysonVershyninZhao2021} proved the MP law for
$d=o(n^{1/3})$. Yaskov~\cite[Theorems~2.2 and~2.3]{Yaskov2023} extended this to
$d=o(\sqrt n)$ under uniformly bounded fourth moments; for i.i.d.\ coordinates
with a fixed law satisfying $\Pp(X^2=1)<1$, this range is necessary. The
obstruction at $d\asymp\sqrt n$ is already visible in the quadratic form with
$A=I$. Write
\[
  R_n=\frac{\|x\|^2}{p},
  \qquad v=\E X^4-1.
\]
If $K$ is the overlap of two independent uniform $d$-subsets of $[n]$, then
\[
  \E K=\frac{d^2}{n},
  \qquad \E R_n=1,
  \qquad \E R_n^2=\E(1+v)^K.
\]
At $d^2/n\to\lambda$, the last expectation tends to $e^{\lambda v}$. In fact, we
prove in \cref{sec:radius} that
\[
  R_n\longrightarrow
  R=\exp\!\left(\sqrt{\lambda v}\,Z-\frac{\lambda v}{2}\right)
  \quad\text{in }W_2,
  \qquad Z\sim N(0,1).
\]
When $\lambda v>0$, concentration of $x^{\mathsf T}Ax/p$ around $\Tr A/p$
therefore fails even for $A=I$. The usual quadratic-form route to the MP
law~\cite{PajorPastur2009,Yaskov2016} cannot apply.

Subtracting the radial contribution repairs this failure:
\[
  \sup_{\|A\|_{\op}\le1}
  \E\left|\frac{x^{\mathsf T}Ax}{p}
  -R_n\frac{\Tr A}{p}\right|^2\longrightarrow0.
\]
Since the limiting radius is strictly positive, this also implies $x^{\mathsf
T}Ax/\|x\|^2-\Tr A/p\to0$ in probability, uniformly over bounded deterministic
$A$, with the ratio defined arbitrarily when $x=0$. Thus the direction still
concentrates in the sense needed for covariance spectra; it need not be
independent of the length.

Together with the radius limit, the estimate gives almost-sure ESD convergence
to the law characterized by
\begin{equation}
  \frac1{s(z)}+z=\E\frac{R}{1+cRs(z)}.
  \label{eq:intro-fixed-point}
\end{equation}
This is the free compound-Poisson law with rate $1/c$ and jump law $\law(cR)$.
The theorem includes $\lambda=0$ and $v=0$, when it reduces to the MP law.

\subsection{Quadratic forms off the radial mode}

The diagonal residual is $p^{-1}\sum_I b_I\prod_{i\in I}X_i^2$ with
$\sum_Ib_I=0$. Expanding in centered products gives its variance as
\[
  p^{-2}\sum_{r=1}^d v^r\|M_r^{\mathsf T}b\|_2^2,
\]
where $M_r$ is the inclusion matrix between $r$-subsets and $d$-subsets. Using
the full operator norm loses the zero-sum condition and gives an order-one bound
at criticality. Its maximizing vector is constant: this is precisely the radial
contribution that has been subtracted. On the orthogonal complement of the
constants, the Johnson-graph spectrum gives an additional factor $r/d$, enough
to make the variance vanish.

The off-diagonal terms have a different difficulty. Group $A_{IJ}x_Ix_J$ by
$D=I\triangle J$. If $\E X^3=0$, distinct groups are orthogonal in $L^2$, but
centering alone does not suffice. Already at degree two, the groups $D=\{1,2\}$
and $D'=\{1,3\}$ contain respectively $X_1X_2X_3^2$ and $X_1X_3X_2^2$, whose
covariance is
\[
  \E[X_1^2X_2^3X_3^3]=(\E X^3)^2.
\]
For example, it equals $1/2$ when $X=\sqrt2$ with probability $1/3$ and
$X=-1/\sqrt2$ with probability $2/3$. To separate these coupled terms, expand
$X^2$ in
\[
  1,\qquad X,\qquad \frac{X^2-1-\mu_3X}{\tau},
  \qquad \mu_3=\E X^3,\quad \tau^2=v-\mu_3^2,
\]
omitting the last function when $\tau=0$. These functions are orthonormal.
Independence of the coordinates then makes their products orthogonal; collecting
coefficients in this basis gives an off-diagonal variance bound of order $d/n$.
The argument loses a constant factor compared with the case $\mu_3=0$. We do not
know whether that loss is necessary.

\subsection{The radius and the limiting law}

The other task is to identify the radius. With $\delta_i=X_i^2-1$, its
elementary-symmetric expansion is close in $L^2$ to a product:
\[
  R_n=\sum_{k=0}^d\frac{(d)_k}{(n)_k}e_k(\delta),
  \qquad P_n=\prod_{i=1}^n\left(1+\frac dn\delta_i\right),
  \qquad R_n-P_n\longrightarrow0.
\]
The approximation holds when $d^2/n$ is bounded. The heuristic
\[
  \log R_n\approx
  \frac dn\sum_i\delta_i-\frac{d^2}{2n^2}\sum_i\delta_i^2
\]
then suggests the lognormal law: the first term has limiting variance $\lambda
v$, and the second tends to $-\lambda v/2$. The proof justifies the logarithmic
expansion for $P_n$ using only $\E\delta_i^2<\infty$, then transfers the limit
to $R_n$.

Equation \eqref{eq:intro-fixed-point} is familiar from weighted covariance
matrices~\cite{PajorPastur2009,ElKaroui2009}. Yaskov's random-profile
replacement theorem also permits the weight to depend on the sample
itself~\cite[Theorem~2.2]{Yaskov2014}. Once the two tensor estimates are proved,
a standard leave-one-out argument yields the equation; we record the version we
use in \cref{thm:abstract-radial}.

The moments of the resulting law are
\[
  \sum_{\pi\in\NC(k)}c^{k-|\pi|}
  \prod_{V\in\pi}\exp\!\left(\lambda v\binom{|V|}{2}\right).
\]
For $\lambda v>0$, their growth makes Carleman's sufficient criterion
inapplicable. We therefore identify the law through its Stieltjes transform
before computing its moments.

The same moment sequence occurs, after changing normalization, in Collins,
Yao and Yuan~\cite[Theorem~2.1 and Remark~2.2]{CollinsYaoYuan2022}. Their
columns are full products of $k$ independent $n$-vectors, with dimension
$n^k$ and $k/n\to\vartheta$. The parameter corresponding to $\lambda v$ is
$\vartheta(\E|\xi|^4-1)$. They leave weak-limit identification open. Our
columns use one base vector and square-free monomials, and the proof of
weak convergence requires only a fourth moment. Yuan~\cite{Yuan2024} treats
the unit-modulus full-product model without restricting $k$.

For principal tensors, the unit-modulus case has deterministic radius. Our
off-diagonal estimate recovers Yaskov's MP range
$\min(d,n-d)=o(n)$~\cite{Yaskov2025}.

Recent work also extends the subcritical theory beyond independent base
coordinates. Cheng and Mikulincer~\cite{ChengMikulincer2026} prove the MP law
for exchangeable unconditional bases under suitable moment assumptions. For
fixed degree, their Theorems~2.3 and~4.2 also treat independent external
weights and an anisotropic version of \eqref{eq:intro-fixed-point}.
Diaconu~\cite{Diaconu2026} treats principal tensors with independent symmetric
sub-Gaussian base coordinates in the range $d=o(\sqrt n)$. Here the base
coordinates remain independent, but the degree reaches the critical scale
and the random weight is the tensor's own squared length.

Free compound-Poisson laws also arise in moving-average and dependent covariance
models~\cite{HasegawaSakumaYoshida2013,Boedihardjo2015}.

The single-vector estimates occupy \cref{sec:radius,sec:johnson,sec:quadratic};
the spectral argument and properties of the limit follow.

\section{Model and main theorem}
\label{sec:model}

For integers $1\le d\le n$, write
\[
  \Omega_{n,d}=\binom{[n]}d,
  \qquad p=p_{n,d}=|\Omega_{n,d}|=\binom nd.
\]
Let $X$ be a real random variable satisfying
\begin{equation}
  \E X=0,
  \qquad \E X^2=1,
  \qquad B:=\E X^4<\infty,
  \label{eq:base-assumptions}
\end{equation}
and define
\[
  v=B-1=\Var(X^2)\ge0.
\]
For every $n$, let $X_1,\ldots,X_n$ be i.i.d. copies of $X$, and define the
principal tensor feature vector $x\in\R^p$ by
\[
  x_I=\prod_{i\in I}X_i,
  \qquad I\in\Omega_{n,d}.
\]
The conditions in \eqref{eq:base-assumptions} imply
\[
  \E x_Ix_J=\1_{\{I=J\}},
\]
so $x$ is isotropic.

Let $x^{(1)},\ldots,x^{(N)}$ be independent copies of $x$ and set
\[
  S_n=\frac1N\sum_{\alpha=1}^{N}
  x^{(\alpha)}x^{(\alpha)\mathsf T}.
\]
We consider the critical asymptotic regime
\begin{equation}
  \frac{d^2}{n}\to\lambda\in[0,\infty),
  \qquad \frac{p}{N}\to c\in(0,\infty).
  \label{eq:critical-scaling}
\end{equation}
All limits in the paper are as $n\to\infty$. In particular $d/n\to0$, $d<n/2$
eventually, and $p=\binom nd\ge n\to\infty$ eventually.

Throughout, $N$ is the number of samples, $R$ (with or without indices) denotes
a radius, and $s$ denotes a Stieltjes transform: for a probability measure $\mu$
on $\R$,
\[
  s_\mu(z)=\int_{\R}\frac1{t-z}\,\mu(\dd t),
  \qquad z\in\C_+.
\]
This convention maps $\C_+$ into $\C_+$.

Define
\[
  R=\exp\!\left(
  \sqrt{\lambda v}\,Z-\frac{\lambda v}{2}
  \right),
  \qquad Z\sim N(0,1),
\]
and let $\nu_{\lambda,v}=\law(R)$.

\begin{theorem}
  \label{thm:main}
  Assume \eqref{eq:base-assumptions} and \eqref{eq:critical-scaling}. Then the
  empirical spectral distribution of $S_n$ converges weakly almost surely to a
  deterministic probability measure $\mu_{c,\lambda,v}$ on $[0,\infty)$. Its
  Stieltjes transform $s(z)$ is characterized as the unique $s(z)\in\C_+$
  satisfying
  \begin{equation}
    1+z s(z)
    =s(z)\int_{0}^{\infty}
    \frac{r}{1+cr s(z)}\,\nu_{\lambda,v}(\dd r),
    \label{eq:main-fixed-point1}
  \end{equation}
  or equivalently
  \begin{equation}
    \frac1{s(z)}+z
    =\E\left[\frac{R}{1+cRs(z)}\right].
    \label{eq:main-fixed-point2}
  \end{equation}
\end{theorem}

The following theorem isolates the covariance argument from the tensor
estimates. It is a random-profile replacement result of the type proved by
Yaskov~\cite{Yaskov2014}, stated here with the scalar profile $R_nI$ and the
uniform $L^2$ hypothesis used below. Its proof is given in
\cref{sec:uniqueness}.

\begin{theorem}[Sample covariance with a random radius]
  \label{thm:abstract-radial}
  Let $\mathbb F\in\{\R,\C\}$ and $p_n\to\infty$. For each $n$, let
  $x_n^{(1)},\ldots,x_n^{(N_n)}$ be independent copies of a random vector
  $x_n\in\mathbb F^{p_n}$, set
  \[
    R_n=\frac{\|x_n\|^2}{p_n},
    \qquad
    W_n=\frac1{N_n}\sum_{\alpha=1}^{N_n}
    x_n^{(\alpha)}{x_n^{(\alpha)}}^*,
  \]
  where $*$ denotes transpose or conjugate transpose according to the field, and
  suppose $p_n/N_n\to c\in(0,\infty)$. Assume that $R_n\Rightarrow\nu$, that
  $\{R_n\}$ is uniformly integrable, and that for some $\eps_n\to0$,
  \begin{equation}
    \sup_{\substack{A\in\C^{p_n\times p_n}\\\|A\|_\op\le1}}
    \E\left|p_n^{-1}x_n^*Ax_n
    -R_np_n^{-1}\Tr A\right|^2\le\eps_n.
    \label{eq:abstract-radial}
  \end{equation}
  Then there is a unique probability measure $\mu_{c,\nu}$ on $[0,\infty)$ whose
  Stieltjes transform satisfies
  \begin{equation}
    \frac1{s(z)}+z=\int_0^\infty\frac{r}{1+crs(z)}\,\nu(\dd r)
    \qquad\text{for all }z\in\C_+,
    \label{eq:abstract-fixed-point}
  \end{equation}
  and $\ESD(W_n)\Rightarrow\mu_{c,\nu}$ weakly in probability. If $\sum_n
  e^{-ap_n}<\infty$ for every $a>0$, the convergence is almost sure.
\end{theorem}

\begin{remark}
  \label{rem:moment-assumptions}
  (i) When $R_n\equiv1$, condition \eqref{eq:abstract-radial} is the good-vector
  condition of Pajor and Pastur~\cite{PajorPastur2009} and a uniform $L^2$ form
  of the quadratic-form condition in Yaskov~\cite{Yaskov2016}. For real vectors,
  Yaskov~\cite[Theorem~2.2]{Yaskov2014} already permits a random positive
  semidefinite profile $\Sigma_n$ coupled to $x_n$. Taking $\Sigma_n=R_nI$ gives
  its quadratic-form hypothesis from \eqref{eq:abstract-radial}, while
  $p_n^{-2}\Tr\Sigma_n^2=R_n^2/p_n\to0$ in probability by tightness of $R_n$. We
  use complex test matrices to apply the estimate conditionally to resolvents.

  (ii) If $v=0$, then $X^2=1$ almost surely, $R\equiv1$, and
  \eqref{eq:main-fixed-point2} is the MP equation. The same is true at the
  included endpoint $\lambda=0$.

  (iii) No condition is imposed on $\E X^3$; \cref{rem:third-moment} describes
  what a nonzero third moment changes in the proof. The limit depends on the law
  of $X$ only through $v=\E X^4-1$; in particular two base laws with equal
  fourth moments and different third moments give the same limit. For $X\sim
  N(0,1)$, $v=2$ and $R=\exp(\sqrt{2\lambda}\,Z-\lambda)$.
\end{remark}

\section{The radius}
\label{sec:radius}

Define the normalized squared radius
\[
  R_n=\frac{\|x\|^2}{p}
  =\frac1{\binom nd}\sum_{I\in\Omega_{n,d}}
  \prod_{i\in I}X_i^2.
\]
Its expansion in centered products gives an $L^2$ comparison with
$\prod_i(1+\tfrac dn(X_i^2-1))$, whose logarithm is a sum of independent terms.
Put $Y_i=X_i^2$ and $\delta_i=Y_i-1$, so that $\E\delta_i=0$ and
$\E\delta_i^2=v$. Let $e_k(\delta)$ denote the $k$th elementary symmetric
polynomial in $\delta_1,\ldots,\delta_n$, with $e_0=1$.

\begin{lemma}
  \label{lem:elementary-expansion}
  For $a_n=d/n$,
  \begin{equation}
    R_n=\sum_{k=0}^{d}
    \frac{(d)_k}{(n)_k}e_k(\delta),
    \label{eq:Rn-expansion}
  \end{equation}
  where $(x)_k=x(x-1)\cdots(x-k+1)$. If
  \[
    P_n=\prod_{i=1}^n(1+a_n\delta_i)
    =\sum_{k=0}^n a_n^k e_k(\delta),
  \]
  then
  \begin{equation}
    \E|R_n-P_n|^2\longrightarrow0.
    \label{eq:Rn-Pn-L2}
  \end{equation}
  Moreover,
  \begin{equation}
    \sup_n\E R_n^2<\infty.
    \label{eq:Rn-second-bound}
  \end{equation}
\end{lemma}

\begin{proof}
  Because
  \[
    \prod_{i\in I}(1+\delta_i)
    =\sum_{T\subseteq I}\prod_{j\in T}\delta_j,
  \]
  summing over all $I\in\Omega_{n,d}$ yields
  \[
    e_d(1+\delta_1,\ldots,1+\delta_n)
    =\sum_{k=0}^d\binom{n-k}{d-k}e_k(\delta).
  \]
  After division by $\binom nd$,
  \[
    \frac{\binom{n-k}{d-k}}{\binom nd}
    =\frac{(d)_k}{(n)_k},
  \]
  which proves \eqref{eq:Rn-expansion}.

  Products indexed by distinct subsets are orthogonal: a coordinate in their
  symmetric difference appears once and has mean zero. Hence $e_k(\delta)$ and
  $e_\ell(\delta)$ are orthogonal for $k\ne\ell$, and
  \[
    \E e_k(\delta)^2=\binom nk v^k.
  \]
  Define
  \[
    \pi_{n,k}=\begin{cases}
      (d)_k/(n)_k, & k\le d, \\
      0,           & k>d.
    \end{cases}
  \]
  Then orthogonality gives
  \begin{equation}
    \E|R_n-P_n|^2
    =\sum_{k=0}^{n}
    (\pi_{n,k}-a_n^k)^2\binom nk v^k.
    \label{eq:orthogonal-sum}
  \end{equation}
  For $k\le d$,
  \[
    \frac{d-j}{n-j}\le\frac dn=a_n,
    \qquad 0\le j\le k-1,
  \]
  so $0\le\pi_{n,k}\le a_n^k$ for every $k$. Thus
  \begin{align}
    (\pi_{n,k}-a_n^k)^2\binom nk v^k
     & \le a_n^{2k}\binom nk v^k\nonumber \\
     & \le \frac1{k!}
    \left(\frac{d^2}{n}v\right)^k.
    \label{eq:domination}
  \end{align}
  The right side is summable uniformly in $n$ because $d^2/n\to\lambda$. If
  $\lambda=0$, the $k=0$ summand in \eqref{eq:orthogonal-sum} vanishes and
  \eqref{eq:domination} gives directly
  \[
    \E|R_n-P_n|^2\le
    \exp\!\left(v\frac{d^2}{n}\right)-1\longrightarrow0.
  \]
  If $\lambda>0$, then $d\to\infty$ and, for each fixed $k$, eventually $k\le d$
  and
  \[
    \frac{\pi_{n,k}}{a_n^k}
    =\prod_{j=0}^{k-1}
    \frac{1-j/d}{1-j/n}\longrightarrow1.
  \]
  Dominated convergence in \eqref{eq:orthogonal-sum} proves \eqref{eq:Rn-Pn-L2}
  in this case as well.

  Finally,
  \[
    \E P_n^2
    =\prod_{i=1}^n\E(1+a_n\delta_i)^2
    =(1+a_n^2v)^n
    \longrightarrow e^{\lambda v}.
  \]
  Together with \eqref{eq:Rn-Pn-L2}, this implies \eqref{eq:Rn-second-bound}.
\end{proof}

\begin{lemma}
  \label{lem:radius-lognormal}
  Under \eqref{eq:critical-scaling},
  \begin{equation}
    \log P_n\Rightarrow
    N\!\left(-\frac{\lambda v}{2},\lambda v\right),
    \label{eq:logPn-limit}
  \end{equation}
  and
  \begin{equation}
    R_n\Rightarrow
    R=\exp\!\left(\sqrt{\lambda v}\,Z-\frac{\lambda v}{2}\right).
    \label{eq:Rn-limit}
  \end{equation}
  In fact $R_n/P_n\to1$ in probability. Moreover,
  \begin{equation}
    \law(R_n)\longrightarrow\law(R)
    \quad\text{in }W_2,
    \qquad
    \E R_n^2\longrightarrow e^{\lambda v}.
    \label{eq:radius-W2}
  \end{equation}
\end{lemma}

\begin{proof}
  Because $Y_i\ge0$ and $a_n<1$ for all sufficiently large $n$,
  \[
    1+a_n\delta_i=1-a_n+a_nY_i\ge1-a_n>0,
  \]
  so $\log P_n$ is well-defined.

  We claim
  \begin{equation}
    \max_{1\le i\le n}|a_n\delta_i|\xrightarrow{p}0.
    \label{eq:max-small}
  \end{equation}
  Since $\E\delta^2<\infty$,
  \[
    t^2\Pp(|\delta|>t)
    \le \E\bigl[\delta^2\1_{\{|\delta|>t\}}\bigr]\to0.
  \]
  Thus, for fixed $\eps>0$,
  \begin{align*}
    \Pp\!\left(\max_i a_n|\delta_i|>\eps\right)
     & \le n\Pp\!\left(|\delta|>\frac{\eps}{a_n}\right) \\
     & =\frac{na_n^2}{\eps^2}
    \left(\frac{\eps}{a_n}\right)^2
    \Pp\!\left(|\delta|>\frac{\eps}{a_n}\right)
    \longrightarrow0,
  \end{align*}
  because $na_n^2=d^2/n\to\lambda$.

  On the event $\max_i|a_n\delta_i|\le1/2$, Taylor's theorem gives
  \[
    \left|\log(1+u)-u+\frac{u^2}{2}\right|
    \le C|u|^3.
  \]
  Consequently,
  \begin{align}
     & \left|
    \log P_n-a_n\sum_{i=1}^n\delta_i
    +\frac{a_n^2}{2}\sum_{i=1}^n\delta_i^2
    \right|\nonumber   \\
     & \hspace{2cm}\le
    C\left(\max_i|a_n\delta_i|\right)
    a_n^2\sum_{i=1}^n\delta_i^2.
    \label{eq:log-taylor}
  \end{align}
  By the law of large numbers,
  \[
    a_n^2\sum_{i=1}^n\delta_i^2
    =\frac{d^2}{n}\left(\frac1n\sum_{i=1}^n\delta_i^2\right)
    \xrightarrow{p}\lambda v.
  \]
  Together with \eqref{eq:max-small}, the right side of \eqref{eq:log-taylor} is
  $o_p(1)$. The classical central limit theorem gives
  \[
    a_n\sum_{i=1}^n\delta_i
    =\frac d{\sqrt n}
    \left(\frac1{\sqrt n}\sum_{i=1}^n\delta_i\right)
    \Rightarrow N(0,\lambda v),
  \]
  and hence \eqref{eq:logPn-limit}.

  By \eqref{eq:logPn-limit}, $P_n\Rightarrow R>0$, so $1/P_n$ is tight. Together
  with \eqref{eq:Rn-Pn-L2}, this gives $R_n/P_n\to1$ in probability and
  \eqref{eq:Rn-limit}.

  For the second moment, let $I,J$ be independent uniform $d$-subsets and
  $K=|I\cap J|$. Since $(1+v)^K=\sum_{r\ge0}(K)_rv^r/r!$ and the hypergeometric
  overlap has factorial moments $\E(K)_r=(d)_r^2/(n)_r$, we have the exact
  identity
  \[
    \E R_n^2=\E(1+v)^K
    =\sum_{r=0}^d\frac{v^r}{r!}\frac{(d)_r^2}{(n)_r}.
  \]
  If $\lambda>0$, every fixed summand converges to $(\lambda v)^r/r!$, and
  \[
    \frac{(d)_r^2}{(n)_r}
    \le\left(\frac{d^2}{n-d}\right)^r
  \]
  gives a summable uniform majorant. If $\lambda=0$, Maclaurin's inequality
  gives
  \[
    1\le \E R_n^2=\E B^K
    \le\left(1+v\frac dn\right)^d
    \le e^{vd^2/n}\longrightarrow1.
  \]
  Thus in all cases $\E R_n^2\to e^{\lambda v}=\E R^2$. Weak convergence plus
  convergence of second moments is equivalent to convergence in $W_2$, proving
  \eqref{eq:radius-W2}. In particular, when $\lambda>0$, $R_n>0$ with
  probability tending to one and $\log R_n-\log P_n\to0$ in probability on that
  event.
\end{proof}

\begin{remark}
  The $W_2$ conclusion is optimal under a fourth-moment assumption. For every
  $q>2$ there is a symmetric standardized $X$ with $\E X^4<\infty$ but
  $\E|X|^{2q}=\infty$. For fixed $I_0\in\Omega_{n,d}$,
  \[
    R_n\ge p^{-1}\prod_{i\in I_0}X_i^2,
  \]
  and therefore $\E R_n^q=\infty$. No universal $W_q$ conclusion with $q>2$ is
  possible under \eqref{eq:base-assumptions}.
\end{remark}

\section{An estimate on the Johnson scheme}
\label{sec:johnson}

For a diagonal quadratic form, subtracting the radial contribution replaces
the coefficients $A_{II}$ by $b_I=A_{II}-p^{-1}\Tr A$, whose sum is zero.
The second-moment kernel of $Y_I=\prod_{i\in I}X_i^2$ is
$\E Y_IY_J=B^{|I\cap J|}$: it depends only on the overlap size. This is the
Johnson-scheme structure behind the diagonal estimate. Expanding $Y_I$ in
centered products leads to inclusion matrices, and we need their norm only
on the orthogonal complement of the constants. The bound below follows from
the classical Johnson-graph spectrum~\cite{BrouwerCohenNeumaier1989} by
factoring these matrices into down-operators.

For $0\le r\le d$, let $M_r$ be the $\binom nd\times\binom nr$ incidence matrix
\[
  (M_r)_{I,T}=\1_{\{T\subseteq I\}},
  \qquad |I|=d,\ |T|=r.
\]

\begin{lemma}
  \label{lem:johnson-spectrum}
  Assume $1\le \ell<n/2$. Let $A_\ell$ be the adjacency matrix of the Johnson
  graph $J(n,\ell)$: two $\ell$-subsets are adjacent when their symmetric
  difference has size two. For a $j$-subset $T$, define
  \[
    f_T(I)=\1_{\{T\subseteq I\}},
    \qquad |I|=\ell,
  \]
  and let $U_j$ be the span of the functions $f_T$ with $|T|=j$. Then
  \[
    U_0\subseteq U_1\subseteq\cdots\subseteq U_\ell.
  \]
  Set $U_{-1}=\{0\}$. If $V_j=U_j\cap U_{j-1}^{\perp}$, then $A_\ell$ acts on
  $V_j$ by the scalar
  \begin{equation}
    \theta_j=(\ell-j)(n-\ell-j)-j,
    \qquad 0\le j\le \ell.
    \label{eq:johnson-eigs}
  \end{equation}
  In particular, when $\ell<n/2$, the largest eigenvalue of $A_\ell$ on the
  orthogonal complement of the constant vector is
  \begin{equation}
    \theta_1=\ell(n-\ell)-n.
    \label{eq:johnson-second}
  \end{equation}
\end{lemma}

\begin{proof}
  For a $(j-1)$-subset $S$,
  \[
    f_S
    =\frac1{\ell-j+1}
    \sum_{\substack{T\supset S\\|T|=j}}f_T,
  \]
  which proves the nesting.

  The constant space $V_0=U_0$ is handled separately:
  $A_\ell\1=\ell(n-\ell)\1=\theta_0\1$. Now fix $1\le j\le \ell$ and a
  $j$-subset $T$. If $T\subset I$, then among the $\ell(n-\ell)$ neighbors of
  $I$, exactly $(\ell-j)(n-\ell)$ still contain $T$. If $|T\cap I|=j-1$, then
  exactly $\ell-j+1$ neighbors contain $T$. Otherwise no neighbor contains $T$.
  Therefore
  \[
    A_\ell f_T
    =(\ell-j)(n-\ell)f_T
    +(\ell-j+1)\1_{\{|T\cap I|=j-1\}}.
  \]
  But
  \[
    \1_{\{|T\cap I|=j-1\}}
    =\sum_{\substack{S\subset T\\|S|=j-1}}f_S(I)-j f_T(I),
  \]
  so
  \begin{equation}
    A_\ell f_T
    =\theta_jf_T
    +(\ell-j+1)
    \sum_{\substack{S\subset T\\|S|=j-1}}f_S,
    \label{eq:Aq-action}
  \end{equation}
  with $\theta_j$ as in \eqref{eq:johnson-eigs}. Since $A_\ell$ is symmetric and
  preserves every $U_j$, it preserves $V_j$, and the lower-level term in
  \eqref{eq:Aq-action} is orthogonal to $V_j$. Thus $A_\ell$ acts by $\theta_j$
  on $V_j$.

  The orthogonal decomposition generated by the nested spaces exhausts
  $U_\ell=\R^{\binom{[n]}\ell}$. Finally,
  \[
    \ell(n-\ell)-\theta_j=j(n-j+1).
  \]
  For $j\ge1$ and $j\le \ell<n/2$,
  \[
    j(n-j+1)-n=(j-1)(n-j)\ge0,
  \]
  with equality at $j=1$. Hence \eqref{eq:johnson-second} is the largest
  nonconstant eigenvalue. It is attained: the nonzero zero-sum function
  $I\mapsto\1_{\{1\in I\}}-\1_{\{2\in I\}}$ belongs to $V_1$.
\end{proof}

\begin{lemma}
  \label{lem:incidence-bound}
  Assume $n>2d$. Let $b=(b_I)_{I\in\Omega_{n,d}}\in\C^{\Omega_{n,d}}$ satisfy
  \[
    \sum_{I\in\Omega_{n,d}}b_I=0.
  \]
  Then for $1\le r\le d$,
  \begin{equation}
    \|M_r^{\mathsf T}b\|_2^2
    \le
    \binom{d-1}{r-1}
    \binom{n-r-1}{d-r}
    \|b\|_2^2.
    \label{eq:incidence-bound}
  \end{equation}
\end{lemma}

\begin{proof}
  The real self-adjoint estimates below extend to complex vectors by
  complexification. For $1\le \ell\le d$, define the down-operator from level
  $\ell$ to level $\ell-1$ by
  \[
    (\partial_\ell f)(S)=\sum_{\substack{I\supset S\\|I|=\ell}}f(I),
    \qquad |S|=\ell-1.
  \]
  A direct count shows
  \begin{equation}
    \partial_\ell^{\mathsf T}\partial_\ell=\ell I+A_\ell.
    \label{eq:down-johnson}
  \end{equation}
  If $f$ has zero sum, then \cref{lem:johnson-spectrum} and
  \eqref{eq:down-johnson} give
  \begin{align}
    \|\partial_\ell f\|_2^2
     & \le (\ell+\theta_1)\|f\|_2^2\nonumber \\
     & =(\ell-1)(n-\ell)\|f\|_2^2.
    \label{eq:down-bound}
  \end{align}
  The zero-sum subspace is preserved, since
  \[
    \sum_{|S|=\ell-1}(\partial_\ell f)(S)
    =\ell\sum_{|I|=\ell}f(I),
  \]
  and
  \[
    \partial_{r+1}\cdots\partial_d b
    =(d-r)!\,M_r^{\mathsf T}b,
  \]
  because every chain from an $r$-subset $T$ to a $d$-subset $I\supset T$ is
  counted $(d-r)!$ times. Applying \eqref{eq:down-bound} successively yields
  \[
    \|M_r^{\mathsf T}b\|_2^2
    \le
    \frac{\prod_{\ell=r+1}^{d}(\ell-1)(n-\ell)}{[(d-r)!]^2}
    \|b\|_2^2.
  \]
  The coefficient equals
  \[
    \binom{d-1}{r-1}
    \binom{n-r-1}{d-r},
  \]
  which proves \eqref{eq:incidence-bound}.
\end{proof}

\section{Radial concentration of quadratic forms}
\label{sec:quadratic}

The diagonal estimate follows from \cref{lem:incidence-bound}. For the
off-diagonal part, grouping by $I\triangle J$ is not enough: in a product of
terms from distinct groups, a coordinate may occur to the third power.
The example in the introduction shows why centering does not make these
groups orthogonal. We instead decompose $X^2$ into its constant part, its
component along $X$, and an orthogonal remainder. Products of these functions
give an orthogonal expansion even when the symmetric-difference groups are
coupled.

Put $\mu_3=\E X^3$, $h_1(X)=X$, and
\[
  \tau^2=v-\mu_3^2\ge0.
\]
The inequality follows from $\mu_3=\operatorname{Cov}(X,X^2)$ and
Cauchy--Schwarz. If $\tau>0$, set
\[
  h_2(X)=\frac{X^2-1-\mu_3X}{\tau},
  \qquad X^2=1+\mu_3h_1(X)+\tau h_2(X).
\]
Then $1,h_1,h_2$ are orthonormal. When $\tau=0$, omit $h_2$ and its term in the
expansion. Independence makes
\[
  \Phi_{S_1,S_2}
  =\prod_{i\in S_1}h_1(X_i)\prod_{i\in S_2}h_2(X_i),
  \qquad S_1\cap S_2=\varnothing,
\]
an orthonormal family; when $\tau=0$, all sums below have $S_2=\varnothing$.

Fix a complex symmetric matrix $A$ with zero diagonal. In an ordered pair $I\ne
J$, write $D=I\triangle J$ and $C=I\cap J$. For $|D|=2q$, $|C|=d-q$, and $C\cap
D=\varnothing$, collect the entries with this difference and overlap:
\[
  \omega_{D,C}
  =\sum_{\substack{U\subset D\\|U|=q}}
  A_{C\cup U,\,C\cup(D\setminus U)}.
\]
Here $1\le q\le d$. The coefficients of the product expansion involve sums over
overlaps containing a prescribed set $W$:
\begin{equation}
  \gamma_D(W)
  =\sum_{\substack{|C|=d-q,\ C\cap D=\varnothing\\C\supseteq W}}
  \omega_{D,C},
  \qquad W\cap D=\varnothing.
  \label{eq:gamma-DW}
\end{equation}
We will also use the Cauchy--Schwarz bound
\begin{equation}
  |\omega_{D,C}|^2
  \le\binom{2q}{q}
  \sum_{\substack{U\subset D\\|U|=q}}
  \left|A_{C\cup U,\,C\cup(D\setminus U)}\right|^2.
  \label{eq:alpha-CS}
\end{equation}

\begin{lemma}[Off-diagonal expansion]
  \label{lem:offdiag-chaos}
  With the notation above,
  \[
    x^{\mathsf T}Ax
    =\sum_{S_1\cap S_2=\varnothing}\zeta_{S_1,S_2}\Phi_{S_1,S_2},
  \]
  where
  \begin{equation}
    \zeta_{S_1,S_2}
    =\sum_{\substack{D\subseteq S_1,\ 2\le |D|\le2d\\|D|\ {\rm even}}}
    \mu_3^{|S_1\setminus D|}\tau^{|S_2|}
    \gamma_D\bigl((S_1\setminus D)\cup S_2\bigr).
    \label{eq:chaos-coefficients}
  \end{equation}
\end{lemma}

\begin{proof}
  Writing $Q=x^{\mathsf T}Ax$ and grouping ordered pairs by $D$ and $C$ gives
  \begin{equation}
    Q=\sum_{q=1}^d\sum_{|D|=2q}
    \left(\prod_{i\in D}h_1(X_i)\right)
    \sum_{\substack{|C|=d-q\\C\cap D=\varnothing}}
    \omega_{D,C}
    \prod_{i\in C}\bigl(1+\mu_3h_1(X_i)+\tau h_2(X_i)\bigr).
    \label{eq:Q-decomp}
  \end{equation}
  Expanding the last product selects disjoint sets $T_1,T_2\subseteq C$ with
  weight $\mu_3^{|T_1|}\tau^{|T_2|}$ and produces $\Phi_{D\cup T_1,T_2}$.
  Collecting terms with $S_1=D\cup T_1$ and $S_2=T_2$ gives
  \eqref{eq:chaos-coefficients}.
\end{proof}

\begin{theorem}[Radial quadratic-form concentration]
  \label{thm:radial-qf}
  Under \eqref{eq:base-assumptions} and \eqref{eq:critical-scaling}, there
  exists a deterministic sequence $\eps_n\to0$ such that for every deterministic
  complex $p\times p$ matrix $A$,
  \begin{equation}
    \E\left|
    \frac{x^{\mathsf T}Ax}{p}
    -R_n\frac{\Tr A}{p}
    \right|^2
    \le \eps_n\|A\|_{\op}^2.
    \label{eq:radial-qf}
  \end{equation}
  For all sufficiently large $n$, one may take
  \[
    \eps_n=\bigl(\sqrt{\eps^{\mathrm{diag}}_n}
    +\sqrt{\eps^{\mathrm{off}}_n}\bigr)^2,
  \]
  where
  \begin{equation}
    \eps^{\mathrm{diag}}_n=\frac{vd}{n-d}e^{vd^2/(n-d)},
    \qquad
    \eps^{\mathrm{off}}_n=
    \frac{8d(n-d)}{n(n-1)}
    \exp\!\left(\frac{2vd^2}{n-d}\right).
    \label{eq:delta-kappa}
  \end{equation}
\end{theorem}

\begin{proof}
  Replacing $A$ by $(A+A^{\mathsf T})/2$ preserves the quadratic form and trace
  and does not increase the operator norm. We may therefore assume that $A$ is
  complex symmetric. Write
  \[
    A=A^{\mathrm{diag}}+A^{\circ},
  \]
  where $A^\circ$ has zero diagonal.

  \emph{Diagonal part.} Put $b_I=A_{II}-p^{-1}\Tr A$,
  $Y_I=\prod_{i\in I}X_i^2$, and $\Delta=\sum_Ib_IY_I$. Then $\sum_Ib_I=0$ and
  \[
    \frac{x^{\mathsf T}A^{\mathrm{diag}}x}{p}
    -R_n\frac{\Tr A}{p}=\frac{\Delta}{p}.
  \]
  Writing $Y_I=\sum_{T\subseteq I}\delta_T$ with
  $\delta_T=\prod_{i\in T}\delta_i$, the zero-sum condition removes the
  constant term and gives
  \[
    \Delta=\sum_{r=1}^d\sum_{|T|=r}(M_r^{\mathsf T}b)_T\delta_T.
  \]
  For distinct subsets $T\ne T'$, $\delta_T$ and $\delta_{T'}$ are orthogonal in
  $L^2$, while
  \[
    \E\delta_T^2=v^{|T|}.
  \]
  Therefore
  \[
    \E|\Delta|^2=\sum_{r=1}^d v^r\|M_r^{\mathsf T}b\|_2^2.
  \]
  By \cref{lem:incidence-bound},
  \[
    \E|\Delta|^2
    \le\|b\|_2^2
    \sum_{r=1}^d v^r
    \binom{d-1}{r-1}
    \binom{n-r-1}{d-r}.
  \]
  Subtracting the mean is an orthogonal projection, so
  \[
    \|b\|_2^2\le\sum_I|A_{II}|^2\le p\|A\|_{\op}^2.
  \]
  The combinatorial identity
  \[
    \frac{
      \binom{d-1}{r-1}\binom{n-r-1}{d-r}
    }{\binom nd}
    =\frac rd\frac{n-d}{n-r}
    \frac{\binom dr^2}{\binom nr}
  \]
  holds, as one checks by expanding the binomial coefficients. Set
  \[
    \Lambda_n=\frac{d^2}{n-d}.
  \]
  The identity, combined with
  \[
    \binom dr\le\frac{d^r}{r!},
    \qquad
    \binom nr\ge\frac{(n-d)^r}{r!}
    \quad (r\le d)
  \]
  and $(n-d)/(n-r)\le1$ gives
  \[
    \frac{
      \binom{d-1}{r-1}\binom{n-r-1}{d-r}
    }p
    \le\frac rd\frac{\Lambda_n^r}{r!}.
  \]
  Hence
  \begin{align}
    \frac{\E|\Delta|^2}{p^2}
     & \le \|A\|_{\op}^2\frac1d
    \sum_{r\ge1}\frac{r(v\Lambda_n)^r}{r!}\nonumber \\
     & =\|A\|_{\op}^2
    \frac{v\Lambda_n}{d}e^{v\Lambda_n}.
    \label{eq:diag-final}
  \end{align}

  \emph{Off-diagonal part.} Let $Q=x^{\mathsf T}A^\circ x$ and use the notation
  of \cref{lem:offdiag-chaos}. Parseval's identity and
  \eqref{eq:chaos-coefficients} give
  \[
    \E|Q|^2=\sum_{S_1\cap S_2=\varnothing}|\zeta_{S_1,S_2}|^2.
  \]
  For each $(S_1,S_2)$, Cauchy--Schwarz over the admissible sets $D\subseteq
  S_1$ yields
  \[
    |\zeta_{S_1,S_2}|^2
    \le 2^{|S_1|}
    \sum_D
    \mu_3^{2|S_1\setminus D|}\tau^{2|S_2|}
    \left|\gamma_D\bigl((S_1\setminus D)\cup S_2\bigr)\right|^2.
  \]
  Fix $D$ and $W\cap D=\varnothing$. The pairs $(S_1,S_2)$ producing this $W$
  correspond exactly to splittings $W=T_1\sqcup T_2$, with $S_1=D\cup T_1$ and
  $S_2=T_2$. Their total weight is
  \begin{align*}
     & \sum_{W=T_1\sqcup T_2}
    2^{|D|+|T_1|}\mu_3^{2|T_1|}\tau^{2|T_2|} \\
     & \qquad=2^{|D|}(2\mu_3^2+\tau^2)^{|W|}
    =2^{|D|}(v+\mu_3^2)^{|W|}
    \le2^{|D|}(2v)^{|W|}.
  \end{align*}
  Consequently,
  \begin{equation}
    \E|Q|^2
    \le\sum_D2^{|D|}
    \sum_{w\ge0}(2v)^w
    \sum_{\substack{|W|=w\\W\cap D=\varnothing}}
    |\gamma_D(W)|^2.
    \label{eq:chaos-parseval-bound}
  \end{equation}

  Fix $D$ with $|D|=2q$ and put $r=d-q$. For $0\le w\le r$, let
  \[
    N_{q,w}=\binom{n-2q-w}{d-q-w}.
  \]
  For a fixed admissible $W$, exactly $N_{q,w}$ sets $C$ in \eqref{eq:gamma-DW}
  contain $W$. Cauchy--Schwarz therefore gives $|\gamma_D(W)|^2\le
  N_{q,w}\sum_{C\supseteq W}|\omega_{D,C}|^2$. After summing over $W$, each
  $r$-set $C$ is counted exactly $\binom rw$ times, so
  \[
    \sum_{\substack{|W|=w\\W\cap D=\varnothing}}|\gamma_D(W)|^2
    \le N_{q,w}\binom rw\sum_C|\omega_{D,C}|^2.
  \]
  Writing $m_q=n-2q$, the ratio has the exact form
  \[
    \frac{N_{q,w}\binom rw}{N_{q,0}}
    =\frac{(r)_w^2}{w!(m_q)_w}
    \le\frac1{w!}
    \left(\frac{r^2}{n-2q-w}\right)^w
    \le\frac1{w!}
    \left(\frac{d^2}{n-d}\right)^w.
  \]
  Here $0\le w\le r=d-q$; the last inequality follows from $n-2q-w\ge n-d-q$ and
  $(d-q)^2/(n-d-q)\le d^2/(n-d)$ when $n\ge2d$ (with the $w=0$ case understood
  directly). Hence
  \[
    \sum_{w\ge0}(2v)^wN_{q,w}\binom rw
    \le N_{q,0}\exp\!\left(\frac{2vd^2}{n-d}\right).
  \]

  Let
  \[
    \|A^{(q)}\|_\F^2
    =\sum_{\substack{I,J\in\Omega_{n,d}\\|I\triangle J|=2q}}|A^\circ_{IJ}|^2
  \]
  and set
  \[
    \rho_q
    =\frac{\binom{2q}{q}\binom{n-2q}{d-q}}{\binom nd}.
  \]
  Cancelling consecutive binomial coefficients gives, for $1\le q<d$,
  \begin{equation}
    \frac{\rho_{q+1}}{\rho_q}
    =\frac{2(2q+1)}{q+1}
    \frac{(d-q)(n-d-q)}{(n-2q)(n-2q-1)}
    \le \beta_n,
    \qquad
    \beta_n=\frac{4dn}{(n-2d)^2}.
    \label{eq:rho-recurrence}
  \end{equation}
  For the inequality, use $2(2q+1)/(q+1)\le4$, $d-q\le d$, $n-d-q\le n$, and
  $(n-2q)(n-2q-1)\ge(n-2d)^2$ when $n>2d$. As $(D,C,U)$ ranges over fixed $q$,
  each ordered pair $(I,J)$ in $A^{(q)}$ occurs once. Combining
  \eqref{eq:chaos-parseval-bound} with \eqref{eq:alpha-CS},
  \[
    \frac{\E|Q|^2}{p^2}
    \le
    \exp\!\left(\frac{2vd^2}{n-d}\right)
    \sum_{q=1}^d4^q\rho_q\frac{\|A^{(q)}\|_\F^2}{p}.
  \]
  Under \eqref{eq:critical-scaling}, $d/n\to0$, hence $\beta_n\to0$ and
  $4\beta_n<1$ eventually. Since
  \[
    \frac{4^{q+1}\rho_{q+1}}{4^q\rho_q}
    =4\frac{\rho_{q+1}}{\rho_q}\le4\beta_n<1,
  \]
  the sequence $4^q\rho_q$ is decreasing for all sufficiently large $n$, and
  \[
    \max_{1\le q\le d}4^q\rho_q
    =4\rho_1=\frac{8d(n-d)}{n(n-1)}.
  \]
  Since the matrices $A^{(q)}$ partition the off-diagonal entries,
  $\sum_q\|A^{(q)}\|_\F^2\le p\|A\|_\op^2$. Therefore
  \[
    \frac{\E|Q|^2}{p^2}\le\eps^{\mathrm{off}}_n\|A\|_\op^2.
  \]
  Combining this with \eqref{eq:diag-final} by Minkowski's inequality gives
  \[
    \left\|
    \frac{x^{\mathsf T}Ax}{p}-R_n\frac{\Tr A}{p}
    \right\|_{L^2}
    \le(\sqrt{\eps^{\mathrm{diag}}_n}+\sqrt{\eps^{\mathrm{off}}_n})\|A\|_\op,
  \]
  which proves \eqref{eq:radial-qf}--\eqref{eq:delta-kappa}.
\end{proof}

If $\E X^3=0$, orthogonality of the symmetric-difference groups gives a sharper
bound. Its specialization to $v=0$ gives \cref{cor:unit-modulus}.

\begin{proposition}[The case $\E X^3=0$]
  \label{prop:sharper-offdiag}
  Assume \eqref{eq:base-assumptions} and $\E X^3=0$, let $A$ be complex
  symmetric, write $A^\circ$ for its zero-diagonal part, and suppose $d/n\to0$.
  Then, for all sufficiently large $n$,
  \begin{equation}
    \frac1{p^2}\E|x^{\mathsf T}A^\circ x|^2
    \le \frac{2d(n-d)}{n(n-1)}
    e^{vd^2/(n-2d)}\|A\|_\op^2.
    \label{eq:sharper-offdiag}
  \end{equation}
\end{proposition}

\begin{proof}
  Let $Q=x^{\mathsf T}A^\circ x$ and use the decomposition \eqref{eq:Q-decomp}.
  When $\mu_3=0$, it takes the form
  \[
    Q=\sum_{q=1}^d\sum_{|D|=2q}
    X_D\sum_{\substack{|C|=d-q\\C\cap D=\varnothing}}\omega_{D,C}Y_C,
    \qquad
    X_D=\prod_{i\in D}X_i,\quad Y_C=\prod_{i\in C}X_i^2.
  \]
  This is the special case of \cref{lem:offdiag-chaos} in which only $D=S_1$
  contributes to each coefficient; we argue directly to get the better constant.
  If $D\ne D'$, the corresponding summands are orthogonal in complex $L^2$.
  Indeed, choose $i\in D\triangle D'$. In the product of a $D$-summand and the
  complex conjugate of a $D'$-summand, the total power of $X_i$ is either one or
  three. The assumptions $\E X=\E X^3=0$ make the expectation zero.

  Fix $D$ with $|D|=2q$, put $m_q=n-2q$ and $r=d-q$, and index $C$ over the
  $r$-subsets of the $m_q$ coordinates outside $D$. Independence and
  $\E X_D^2=1$ remove the factor $X_D$ from the second moment. On the remaining
  coordinates,
  \[
    \E(Y_CY_{C'})=B^{|C\cap C'|}.
  \]
  The second-moment matrix
  \[
    \Sigma_{C,C'}=B^{|C\cap C'|}
  \]
  is symmetric and entrywise nonnegative, with constant row sum
  \[
    \binom{m_q}{r}\E B^K,
  \]
  where $K$ is the intersection size of a fixed $r$-subset and a uniformly
  random $r$-subset of an $m_q$-set. Therefore its spectral norm is at most this
  row sum (indeed, the constant vector is a Perron eigenvector).

  To bound $\E B^K$, view $B^K$ as the product of $r$ numbers sampled without
  replacement from a population containing $r$ copies of $B$ and $m_q-r$ copies
  of $1$. Maclaurin's inequality for elementary symmetric means yields
  \[
    \E B^K
    \le\left(1+\frac{(B-1)r}{m_q}\right)^r
    \le\exp\!\left(\frac{vr^2}{m_q}\right)
    \le\exp\!\left(\frac{vd^2}{n-2d}\right).
  \]
  In \eqref{eq:alpha-CS}, each choice $(D,C,U)$ specifies exactly one ordered
  pair $(I,J)$ with $I\triangle J=D$ and $I\cap J=C$. Summing over these choices
  and using orthogonality in $D$ gives
  \begin{align*}
    \E|Q|^2
     & \le e^{vd^2/(n-2d)}
    \max_{1\le q\le d}
    \left[
      \binom{2q}{q}\binom{n-2q}{d-q}
      \right]
    \|A^{\circ}\|_\F^2.
  \end{align*}
  The ratio of the combinatorial factor to $p$ is
  \begin{equation}
    \frac{\binom{2q}{q}\binom{n-2q}{d-q}}{\binom nd}
    =\binom{2q}{q}
    \frac{(d)_q(n-d)_q}{(n)_{2q}}.
    \label{eq:offdiag-ratio}
  \end{equation}
  Denote the ratio in \eqref{eq:offdiag-ratio} by $\rho_q$, consistently with
  the notation above. By \eqref{eq:rho-recurrence}, if $d/n\to0$ then eventually
  $\beta_n<1$. Thus $\rho_q$ is decreasing and
  \[
    \max_{1\le q\le d}\rho_q=\rho_1
    =\frac{2d(n-d)}{n(n-1)}.
  \]
  Also
  \[
    \|A^{\circ}\|_\F^2
    \le\|A\|_\F^2
    \le p\|A\|_{\op}^2.
  \]
  Hence
  \[
    \frac{\E|Q|^2}{p^2}
    \le
    \frac{2d(n-d)}{n(n-1)}
    e^{vd^2/(n-2d)}\|A\|_{\op}^2.
  \]
  The right-hand side tends to zero whenever
  \[
    \frac dn\exp\!\left(\frac{vd^2}{n-2d}\right)\longrightarrow0.
  \]
  When $v=0$, this holds throughout $d/n\to0$.
\end{proof}

\begin{remark}
  \label{rem:third-moment}
  Centering alone does not give the block orthogonality used in
  \cref{prop:sharper-offdiag}. Let $\Pp(X=\sqrt2)=1/3$ and
  $\Pp(X=-1/\sqrt2)=2/3$. Then $\E X=0$, $\E X^2=1$, but $\E X^3=1/\sqrt2$. The
  monomials corresponding to $(D,C)=(\{1,2\},\{3\})$ and
  $(D',C')=(\{1,3\},\{2\})$ have covariance
  \[
    \E[X_1^2X_2^3X_3^3]=(\E X^3)^2=\frac12.
  \]
  More generally, for $D=I\triangle J$, $D'=I'\triangle J'$, $C=I\cap J$, and
  $C'=I'\cap J'$, independence gives the exact formula
  \[
    \E[x_Ix_Jx_{I'}x_{J'}]
    =\1_{\{D\setminus D'\subset C',\ D'\setminus D\subset C\}}
    (\E X^3)^{|D\triangle D'|}B^{|C\cap C'|}.
  \]
  Indeed, coordinates in $D\triangle D'$ occur to the first power unless the
  displayed containments hold, in which case they occur to the third power;
  coordinates in $D\cap D'$ and the remaining coordinates of $C\triangle C'$
  occur to the second power, and those in $C\cap C'$ to the fourth.
\end{remark}

\section{Resolvent analysis and the covariance principle}
\label{sec:uniqueness}

The tensor estimates are now available. To pass from one sample to the
covariance matrix, the quadratic-form estimate must be applied to a random
leave-one-out resolvent. Independence allows this by conditioning; a lower
bound on the Sherman--Morrison denominator then controls the replacement
error. We prove the argument under the hypotheses of
\cref{thm:abstract-radial}, so it also applies to the weighted Gaussian model
used later to compute moments.

For a Hermitian matrix $H$ and $z=E+i\eta\in\C_+$ write $G_H(z)=(H-zI)^{-1}$.

\begin{lemma}
  \label{lem:denominator}
  Let $t_1,\ldots,t_k\ge0$ and $w_1,\ldots,w_k\ge0$, and put $u(z)=\sum_{j=1}^k
  w_j/(t_j-z)$ for $z=E+i\eta\in\C_+$. Then
  \[
    |1+u(z)|\ge\frac{\eta}{|z|}.
  \]
\end{lemma}

\begin{proof}
  For $t\ge0$, $\Imc\frac{z}{t-z}=\frac{t\eta}{(t-E)^2+\eta^2}\ge0$, so
  $\Imc[z(1+u(z))]=\eta+\sum_jw_j\Imc\frac{z}{t_j-z}\ge\eta$, and
  $|z|\,|1+u(z)|\ge\eta$.
\end{proof}

\begin{lemma}
  \label{lem:rank-one-trace}
  Let $H,H'$ be Hermitian $p\times p$ matrices with
  $\operatorname{rank}(H-H')\le1$, and $z=E+i\eta\in\C_+$. Then
  \[
    \left|\frac1p\Tr G_H(z)-\frac1p\Tr G_{H'}(z)\right|
    \le\frac{\pi}{p\eta}.
  \]
\end{lemma}

\begin{proof}
  By the rank inequality \cite[Theorem~A.43]{BaiSilverstein2010}, the spectral
  distribution functions satisfy $\sup_t|F_H(t)-F_{H'}(t)|\le1/p$. Integrating
  by parts, $\bigl|\int\frac{\dd(F_H-F_{H'})(t)}{t-z}\bigr|
  =\bigl|\int\frac{(F_H-F_{H'})(t)}{(t-z)^2}\,\dd t\bigr|
  \le\frac1p\int_\R\frac{\dd t}{|t-z|^2}=\frac{\pi}{p\eta}$.
\end{proof}

\begin{lemma}
  \label{lem:uniqueness}
  Let $\nu$ be a probability measure on $[0,\infty)$, $c>0$, and $z\in\C_+$. The
  equation
  \begin{equation}
    \frac1s+z
    =\int_0^\infty\frac{r}{1+crs}\,\nu(\dd r)
    \label{eq:general-fixed}
  \end{equation}
  has at most one solution $s\in\C_+$.
\end{lemma}

\begin{proof}
  Let $s=u+iy\in\C_+$ solve \eqref{eq:general-fixed}, $z=E+i\eta$. Taking
  imaginary parts, $-y/|s|^2+\eta=-cy\int r^2|1+crs|^{-2}\,\nu(\dd r)$, that is,
  \begin{equation}
    c|s|^2\int\frac{r^2}{|1+crs|^2}\,\nu(\dd r)
    =1-\frac{\eta|s|^2}{y}<1.
    \label{eq:stability-strict}
  \end{equation}
  If $s_1\ne s_2$ are two solutions, subtracting the equations and dividing by
  $s_1-s_2$ gives
  \[
    \frac1{s_1s_2}
    =c\int\frac{r^2}{(1+crs_1)(1+crs_2)}\,\nu(\dd r),
  \]
  and Cauchy--Schwarz bounds the modulus of the right side by
  $|s_1|^{-1}|s_2|^{-1}$ times the geometric mean of the two quantities in
  \eqref{eq:stability-strict}, each of which is less than one. This is a
  contradiction.
\end{proof}

\begin{lemma}
  \label{lem:as-upgrade}
  Let $y_{n,1},\ldots,y_{n,N_n}\in\mathbb F^{p_n}$ be independent, $\mathbb
  F\in\{\R,\C\}$, and
  $W_n=\frac1{N_n}\sum_{\alpha=1}^{N_n}y_{n,\alpha}y_{n,\alpha}^*$. Suppose
  $p_n/N_n\to c\in(0,\infty)$ and $\sum_ne^{-ap_n}<\infty$ for every $a>0$. If
  $\ESD(W_n)$ converges weakly in probability to a deterministic probability
  measure $\mu$, then it converges almost surely.
\end{lemma}

\begin{proof}
  Let $\ell_n$ be the L\'evy distance between $\ESD(W_n)$ and $\mu$. Changing
  one sample changes $W_n$ by a matrix of rank at most two, hence, by the rank
  inequality \cite[Theorem~A.43]{BaiSilverstein2010}, changes $\ESD(W_n)$ by at
  most $2/p_n$ in Kolmogorov distance, and therefore changes $\ell_n$ by at most
  $2/p_n$. McDiarmid's inequality gives
  $\Pp(\ell_n-\E\ell_n>t)\le\exp(-t^2p_n^2/(2N_n))$. Since $0\le\ell_n\le1$ and
  $\ell_n\to0$ in probability, $\E\ell_n\to0$. For fixed $\eps>0$, the bound
  with $t=\eps/2$ is eventually at most $e^{-a_\eps p_n}$ for some
  $a_\eps>0$. Borel--Cantelli gives $\ell_n\to0$ almost surely, without any
  independence assumption between different $n$.
\end{proof}

\begin{proof}[Proof of \cref{thm:abstract-radial}]
  Write $R_{n,\alpha}=\|x_n^{(\alpha)}\|^2/p_n$ and
  $\widehat\nu_n=\frac1{N_n}\sum_{\alpha=1}^{N_n}\delta_{R_{n,\alpha}}$.

  \emph{Empirical radii.} For bounded continuous $f$, independence gives
  \[
    \Var\!\left(\frac1{N_n}\sum_{\alpha=1}^{N_n}f(R_{n,\alpha})\right)
    \le \frac{\|f\|_\infty^2}{N_n}\longrightarrow0,
  \]
  and $\E f(R_n)\to\int f\,\dd\nu$ because $R_n\Rightarrow\nu$. Hence the
  empirical averages converge in probability to $\int f\,\dd\nu$, and applying
  this to a countable convergence-determining family of bounded continuous
  functions on $[0,\infty)$ gives $\widehat\nu_n\Rightarrow\nu$ weakly in
  probability. For the first moment, fix $L>0$. The same variance bound applies
  to the bounded variables $R_{n,\alpha}\wedge L$, so their empirical mean
  differs from $\E(R_n\wedge L)$ by $o_p(1)$. The empirical tail
  $\frac1{N_n}\sum_\alpha R_{n,\alpha}\1_{\{R_{n,\alpha}>L\}}$ has expectation
  $\E[R_n\1_{\{R_n>L\}}]$, which by uniform integrability is small uniformly in
  $n$ when $L$ is large, so by Markov's inequality the empirical tail is small
  in probability. Uniform integrability also gives
  $\E R_n\to\int r\,\nu(\dd r)<\infty$. Letting $L\to\infty$,
  \begin{equation}
    \frac1{N_n}\sum_{\alpha=1}^{N_n}R_{n,\alpha}
    \xrightarrow{p}\int r\,\nu(\dd r).
    \label{eq:abstract-radius-mean}
  \end{equation}
  The first spectral moment is $p_n^{-1}\Tr W_n=N_n^{-1}\sum_\alpha
  R_{n,\alpha}$. Hence, for $M>0$,
  \[
    \ESD(W_n)([M,\infty))\le\frac{\Tr W_n}{Mp_n}.
  \]
  Together with \eqref{eq:abstract-radius-mean}, this proves tightness in
  probability.

  \emph{The resolvent equation.} Fix $z=E+i\eta\in\C_+$ and put $c_n=p_n/N_n$,
  $G_n=G_{W_n}(z)$ and $s_n(z)=\frac1{p_n}\Tr G_n$. For each $\alpha$ set
  \[
    W_n^{(\alpha)}=W_n-\frac1{N_n}x_n^{(\alpha)}{x_n^{(\alpha)}}^*,
    \qquad
    G_n^{(\alpha)}=G_{W_n^{(\alpha)}}(z),
    \qquad
    s_n^{(\alpha)}(z)=\frac1{p_n}\Tr G_n^{(\alpha)},
  \]
  \[
    a_{n,\alpha}=\frac1{p_n}{x_n^{(\alpha)}}^*G_n^{(\alpha)}x_n^{(\alpha)}.
  \]
  The identity $W_nG_n=I+zG_n$ gives
  $1+zs_n=\frac1{p_nN_n}\sum_\alpha{x_n^{(\alpha)}}^*G_nx_n^{(\alpha)}$, and the
  Sherman--Morrison formula gives
  ${x_n^{(\alpha)}}^*G_nx_n^{(\alpha)}=p_na_{n,\alpha}/(1+c_na_{n,\alpha})$.
  Hence
  \[
    1+zs_n(z)
    =\frac1{N_n}\sum_{\alpha=1}^{N_n}
    \frac{a_{n,\alpha}}{1+c_na_{n,\alpha}}.
  \]

  We control the denominators with \cref{lem:denominator}. Diagonalize
  $W_n^{(\alpha)}=\sum_jt_ju_ju_j^*$ with $t_j\ge0$. Then
  \[
    c_na_{n,\alpha}=\frac1{N_n}\sum_j\frac{|u_j^*x_n^{(\alpha)}|^2}{t_j-z},
    \qquad
    c_nrs_n^{(\alpha)}(z)=\frac r{N_n}\sum_j\frac1{t_j-z}
    \quad(r\ge0),
  \]
  and $c_nrs_n(z)$ has the same form with the eigenvalues of $W_n$. All three
  are of the type considered in \cref{lem:denominator}, so
  \begin{equation}
    |1+c_na_{n,\alpha}|^{-1},\quad
    |1+c_nrs_n^{(\alpha)}(z)|^{-1},\quad
    |1+c_nrs_n(z)|^{-1}
    \ \le\ \frac{|z|}{\eta}
    \qquad(r\ge0).
    \label{eq:abstract-denominators}
  \end{equation}

  We first replace $a_{n,\alpha}$ by $R_{n,\alpha}s_n^{(\alpha)}(z)$. The
  resolvent $G_n^{(\alpha)}$ is independent of $x_n^{(\alpha)}$ and
  $\|G_n^{(\alpha)}\|_\op\le1/\eta$, so applying \eqref{eq:abstract-radial}
  conditionally on $G_n^{(\alpha)}$ to the matrix $\eta G_n^{(\alpha)}$ gives,
  uniformly in $\alpha$,
  \begin{equation}
    \E\left|a_{n,\alpha}-R_{n,\alpha}s_n^{(\alpha)}(z)\right|^2
    \le\frac{\eps_n}{\eta^2}.
    \label{eq:abstract-loo-qf}
  \end{equation}
  For any $a,b$ with $1+c_na\ne0\ne1+c_nb$,
  \[
    \frac a{1+c_na}-\frac b{1+c_nb}
    =\frac{a-b}{(1+c_na)(1+c_nb)},
  \]
  so by \eqref{eq:abstract-denominators}, with
  $b=R_{n,\alpha}s_n^{(\alpha)}(z)$,
  \[
    \left|
    \frac{a_{n,\alpha}}{1+c_na_{n,\alpha}}
    -\frac{R_{n,\alpha}s_n^{(\alpha)}(z)}{1+c_nR_{n,\alpha}s_n^{(\alpha)}(z)}
    \right|
    \le\left(\frac{|z|}{\eta}\right)^2
    \bigl|a_{n,\alpha}-R_{n,\alpha}s_n^{(\alpha)}(z)\bigr|.
  \]
  Averaging over $\alpha$ and using Cauchy--Schwarz with
  \eqref{eq:abstract-loo-qf} bounds the expected error by
  $|z|^2\sqrt{\eps_n}/\eta^3$, which tends to zero. Thus
  \begin{equation}
    1+zs_n(z)
    =\frac1{N_n}\sum_{\alpha=1}^{N_n}
    \frac{R_{n,\alpha}s_n^{(\alpha)}(z)}{1+c_nR_{n,\alpha}s_n^{(\alpha)}(z)}
    +o_p(1).
    \label{eq:abstract-first-replacement}
  \end{equation}

  Next we replace $s_n^{(\alpha)}$ by $s_n$. For $r\ge0$ and $w_1,w_2$ with
  $|1+c_nrw_i|^{-1}\le|z|/\eta$, the same algebraic identity gives
  \[
    \left|\frac{rw_1}{1+c_nrw_1}-\frac{rw_2}{1+c_nrw_2}\right|
    =\frac{r|w_1-w_2|}{|1+c_nrw_1|\,|1+c_nrw_2|}
    \le r\left(\frac{|z|}{\eta}\right)^2|w_1-w_2|.
  \]
  Since $W_n-W_n^{(\alpha)}$ has rank one, \cref{lem:rank-one-trace} gives
  $|s_n^{(\alpha)}(z)-s_n(z)|\le\pi/(p_n\eta)$ for every $\alpha$. Hence
  \[
    \frac1{N_n}\sum_{\alpha=1}^{N_n}
    \left|
    \frac{R_{n,\alpha}s_n^{(\alpha)}(z)}{1+c_nR_{n,\alpha}s_n^{(\alpha)}(z)}
    -\frac{R_{n,\alpha}s_n(z)}{1+c_nR_{n,\alpha}s_n(z)}
    \right|
  \]
  is at most $(|z|/\eta)^2\,\pi(p_n\eta)^{-1}N_n^{-1}\sum_\alpha R_{n,\alpha}$,
  which is $o_p(1)$ by \eqref{eq:abstract-radius-mean}. Substituting in
  \eqref{eq:abstract-first-replacement},
  \begin{equation}
    1+zs_n(z)
    =s_n(z)\frac1{N_n}\sum_{\alpha=1}^{N_n}
    \frac{R_{n,\alpha}}{1+c_nR_{n,\alpha}s_n(z)}+o_p(1).
    \label{eq:abstract-approx-fixed-point}
  \end{equation}

  \emph{Identification of the limit.} Take an arbitrary subsequence. Since
  $\widehat\nu_n\Rightarrow\nu$, \eqref{eq:abstract-radius-mean}, and
  \eqref{eq:abstract-approx-fixed-point} for each fixed $z$ all hold in
  probability, a diagonal argument over a fixed countable dense set $\mathcal
  Z\subset\C_+$ yields a further subsequence along which, almost surely, all of
  these statements hold simultaneously for every $z\in\mathcal Z$. Fix a sample
  point in this event. Along the subsequence $\sup_np_n^{-1}\Tr W_n<\infty$ by
  \eqref{eq:abstract-radius-mean}, so the measures $\ESD(W_n)$ are tight, and
  every further subsequence has a weakly convergent subsubsequence. Let $\mu$ be
  one of its limits; then $s_n(z)\to s_\mu(z)$ for every $z\in\C_+$ along that
  subsubsequence.

  Fix $z\in\mathcal Z$ and put $g_n=c_ns_n(z)$, $g=cs_\mu(z)$, so that $g_n\to
  g\in\C_+$ and $\Imc g_n\ge(\Imc g)/2$ for large $n$. For every $r\ge0$,
  \[
    \left|\frac{r}{1+g_nr}-\frac{r}{1+gr}\right|
    =\frac{|g_n-g|r^2}{|1+g_nr|\,|1+gr|}
    \le\frac{|g_n-g|}{(\Imc g_n)(\Imc g)},
  \]
  so $r\mapsto r/(1+c_nrs_n(z))$ converges to $r\mapsto r/(1+crs_\mu(z))$
  uniformly on $[0,\infty)$. The limit is bounded and continuous, and
  $\widehat\nu_n\Rightarrow\nu$, hence
  \[
    \frac1{N_n}\sum_{\alpha=1}^{N_n}
    \frac{R_{n,\alpha}}{1+c_nR_{n,\alpha}s_n(z)}
    \longrightarrow
    \int_0^\infty\frac{r}{1+crs_\mu(z)}\,\nu(\dd r).
  \]
  Passing to the limit in \eqref{eq:abstract-approx-fixed-point},
  \[
    \frac1{s_\mu(z)}+z
    =\int_0^\infty\frac{r}{1+crs_\mu(z)}\,\nu(\dd r),
    \qquad z\in\mathcal Z.
  \]
  By \cref{lem:uniqueness}, $s_\mu(z)$ is determined by this equation for each
  $z\in\mathcal Z$, so any two cluster measures have Stieltjes transforms that
  agree on $\mathcal Z$, hence on $\C_+$, and coincide. Thus along the original
  subsequence a further subsequence converges almost surely to a single
  deterministic measure $\mu_{c,\nu}$, which does not depend on the subsequence.
  By the subsequence characterization of convergence in probability,
  $\ESD(W_n)\Rightarrow\mu_{c,\nu}$ in probability. The identity
  \eqref{eq:abstract-fixed-point} holds on $\mathcal Z$ and extends to $\C_+$ by
  analyticity, and uniqueness of $\mu_{c,\nu}$ follows from
  \cref{lem:uniqueness} and the fact that a Stieltjes transform determines the
  measure. The almost-sure statement follows from \cref{lem:as-upgrade}.
\end{proof}

\section{Proof of the main theorem}
\label{sec:main-proof}

\begin{proof}[Proof of \cref{thm:main}]
  We check the hypotheses of \cref{thm:abstract-radial} for $x_n=x$, $N_n=N$,
  $p_n=p$. By \cref{lem:radius-lognormal}, the laws of $R_n$ converge to
  $\nu_{\lambda,v}$ in $W_2$, so the radii are uniformly integrable.
  \Cref{thm:radial-qf} gives \eqref{eq:abstract-radial}, and $p/N\to c$ by
  assumption. Hence $\ESD(S_n)$ converges in probability to the measure
  determined by \eqref{eq:main-fixed-point2}, and the solution is unique in
  $\C_+$.

  Under \eqref{eq:critical-scaling}, eventually $1\le d<n$, so $p=\binom nd\ge
  n$ and $\sum_ne^{-ap}<\infty$ for every $a>0$. The convergence is therefore
  almost sure.
\end{proof}

\section{The limit law}
\label{sec:consequences}

\subsection{Free compound-Poisson structure and moments}

Let $G(z)=-s(z)$ be the usual Cauchy transform. Rewriting
\eqref{eq:main-fixed-point2} gives
\begin{equation}
  z=\frac1{G(z)}+\mathcal R(G(z)),
  \qquad
  \mathcal R(w)=\E\frac{R}{1-cRw},
  \quad w\in\C_-.
  \label{eq:R-transform}
\end{equation}
\begin{proposition}
  \label{prop:free-Poisson-moments}
  The limiting measure $\mu_{c,\lambda,v}$ is the free compound-Poisson law with
  rate $1/c$ and jump distribution $\law(cR)$. For every $k\ge1$ it has the
  finite moment
  \begin{equation}
    M_k=\int t^k\,\mu_{c,\lambda,v}(\dd t)
    =\sum_{\pi\in\NC(k)}c^{k-|\pi|}
    \prod_{V\in\pi}\E R^{|V|}.
    \label{eq:NC-moments}
  \end{equation}
  Equivalently, its free cumulants are
  \begin{equation}
    \kappa_k=c^{k-1}\E R^k
    =c^{k-1}\exp\!\left(\frac{k(k-1)}2\lambda v\right).
    \label{eq:free-cumulants}
  \end{equation}
\end{proposition}

\begin{proof}
  Let $\pi$ be the free compound-Poisson law with rate $1/c$ and jump law
  $\law(cR)$. Its Voiculescu transform is
  \[
    \phi_\pi(\zeta)=\E\frac{R\zeta}{\zeta-cR}.
  \]
  This formula holds for unbounded jump measures as well
  \cite[Equation~(2.14)]{HasebeSakuma2017}. On a domain near infinity where
  $F_\pi=1/G_\pi$ has an analytic inverse, the identity
  $F_\pi^{-1}(\zeta)=\zeta+\phi_\pi(\zeta)$ gives \eqref{eq:R-transform}
  for $G_\pi$
  \cite{BercoviciVoiculescu1993}. By \cref{lem:uniqueness}, $-G_\pi=s$ on their
  common domain, and hence throughout $\C_+$ by analytic continuation. This
  identifies the law without a moment-determinacy assumption.

  To verify its moments, use a Gaussian model with independent radial weights.
  Let $z_1,\ldots,z_N\in\C^p$ have independent standard circular complex
  Gaussian entries, normalized by $\E|z_{\alpha j}|^2=1$, and let
  $R_1,\ldots,R_N$ be independent variables with law $\nu_{\lambda,v}$,
  independent of the Gaussian vectors. Put
  \[
    W_{p,N}=\frac1N\sum_{\alpha=1}^NR_\alpha z_\alpha z_\alpha^*.
  \]
  For $y=\sqrt R\,z$, its normalized squared radius is $\widehat
  R_p=R\|z\|^2/p$, and
  \[
    \E|\widehat R_p-R|^2=\frac{\E R^2}{p}.
  \]
  Thus $\widehat R_p\to R$ in $L^2$ and the radii are uniformly integrable. For
  every deterministic $A$, circular Gaussian contraction gives
  \begin{align*}
     & \E\left|
    \frac{y^*Ay}{p}
    -\widehat R_p\frac{\Tr A}{p}
    \right|^2   \\
     & \qquad=
    \frac{\E R^2}{p^2}
    \Tr\!\left[
      \left(A-\frac{\Tr A}{p}I\right)
      \left(A-\frac{\Tr A}{p}I\right)^*
      \right]
    \le\frac{\E R^2}{p}\|A\|_{\op}^2.
  \end{align*}
  By \cref{thm:abstract-radial}, $\ESD(W_{p,N})\Rightarrow\mu_{c,\lambda,v}$ in
  probability when $p/N\to c$.

  Fix $k\ge1$ and let $\gamma=(1\,2\,\cdots\,k)$. Expand $p^{-1}\Tr W_{p,N}^k$
  over the sample labels and Gaussian coordinate labels. A Wick contraction
  $\sigma\in S_k$ contributes $p^{\#(\gamma\sigma)}$. A partition $\tau$ of its
  cycles records coincidences among the sample labels. Different cycles may
  carry the same label, so these partitions must be retained in the finite
  formula. Assigning distinct labels to the blocks of $\tau$ gives
  $(N)_{|\tau|}$. Thus
  \begin{equation}
    \E\frac1p\Tr W_{p,N}^k
    =\frac1{pN^k}\sum_{\sigma\in S_k}p^{\#(\gamma\sigma)}
    \sum_{\tau\in\mathcal P(\operatorname{cyc}\sigma)}
    (N)_{|\tau|}
    \prod_{\beta\in\tau}
    \E R^{\sum_{C\in\beta}|C|}.
    \label{eq:exact-Wick}
  \end{equation}
  Here $\mathcal P(\operatorname{cyc}\sigma)$ is the set of partitions of the
  cycles of $\sigma$. The discrete partition contributes
  \[
    N^{\#\sigma}\prod_{C\in\operatorname{cyc}\sigma}\E R^{|C|}
    +O_{k,R}(N^{\#\sigma-1}),
  \]
  where the constant is finite because $R$ has moments of every order. The genus
  inequality
  \[
    \#(\gamma\sigma)+\#\sigma\le k+1
  \]
  shows that only equality cases survive. These geodesic permutations are in
  bijection with $\NC(k)$, their cycles being the blocks of the corresponding
  partition. Since $\#(\gamma\sigma)-1=k-\#\sigma$ in the equality case,
  \eqref{eq:exact-Wick} converges to the right side of \eqref{eq:NC-moments}.

  The same genus bound in \eqref{eq:exact-Wick} at order $2k$, including
  collision partitions, gives $\sup_{p,N}\E[p^{-1}\Tr W_{p,N}^{2k}]<\infty$
  along the sequence. Hence
  \[
    \E\int_{t>L}t^k\,\ESD(W_{p,N})(\dd t)
    \le L^{-k}\E\int t^{2k}\,\ESD(W_{p,N})(\dd t)
    \le C_kL^{-k}.
  \]
  Apply ESD convergence to the bounded continuous function $(t\wedge L)^k$, then
  let $L\to\infty$ using this tail bound. This proves \eqref{eq:NC-moments}; the
  moment--cumulant formula \cite{NicaSpeicher2006} gives
  \eqref{eq:free-cumulants}.
\end{proof}

In particular,
\begin{align*}
  M_1 & =1,                                     \\
  M_2 & =1+c\,e^{\lambda v},                    \\
  M_3 & =1+3c\,e^{\lambda v}+c^2e^{3\lambda v}.
\end{align*}
For comparison, the second moment of the MP law is $1+c$. These are moments of
the limit law; under \eqref{eq:base-assumptions} alone the spectral moments of
$S_n$ itself need not converge.

\begin{remark}
  The feature-side law is generally not $\nu_{\lambda,v}\boxtimes\mathrm{MP}_c$
  under the usual mean-one convention. Indeed, writing $m_j=\int
  r^j\,\nu_{\lambda,v}(\dd r)$, the second moment of the present law is
  $m_1^2+cm_2$, whereas the second moment of
  $\nu_{\lambda,v}\boxtimes\mathrm{MP}_c$ is $m_2+cm_1^2$.

  There is an exact companion formulation. Let $\underline\mu$ be the limiting
  ESD of the $N\times N$ companion covariance and let $\operatorname{Dil}_c$
  denote dilation by $c$. Then
  \[
    \underline\mu=(1-c)\delta_0+c\mu_{c,\lambda,v}
    =\operatorname{Dil}_c\!\left(
    \nu_{\lambda,v}\boxtimes\mathrm{MP}_{1/c}\right).
  \]
  For $c>1$, the zero atom of $\mu_{c,\lambda,v}$ cancels the negative
  coefficient of $\delta_0$.
\end{remark}

\subsection{Atoms}

Although \eqref{eq:main-fixed-point2} was derived on $\C_+$, the Stieltjes
transform is holomorphic on $\C\setminus[0,\infty)$ and the identity extends by
conjugation. At $z=-\sigma<0$, it also follows directly by continuity from
$\C_+$: $s(-\sigma)>0$, so all denominators are bounded away from zero.

The atom can also be read off from general results on freely infinitely
divisible laws, see \cite[Theorem~3.4]{HasebeSakuma2017}; the direct argument is
short.

\begin{proposition}
  \label{prop:atoms}
  The limiting law has no atoms in $(0,\infty)$, and
  \begin{equation}
    \mu_{c,\lambda,v}(\{0\})=\left(1-\frac1c\right)_+.
    \label{eq:zero-atom}
  \end{equation}
\end{proposition}

\begin{proof}
  For $\sigma>0$, put $u_\sigma=\sigma s(-\sigma)$. The fixed-point equation
  gives
  \begin{equation}
    1-u_\sigma=\frac1c\E\frac{cRs(-\sigma)}{1+cRs(-\sigma)}.
    \label{eq:zero-atom-eq}
  \end{equation}
  As $\sigma\downarrow0$, $u_\sigma\to a:=\mu(\{0\})$. If $c>1$, the integrand
  is at most one, so $a\ge1-1/c>0$ and consequently $s(-\sigma)\to\infty$. Since
  $R>0$ almost surely, dominated convergence in \eqref{eq:zero-atom-eq} gives
  $1-a=1/c$. If $c\le1$ and $a>0$, the same argument would give $a=1-1/c\le0$, a
  contradiction. This proves \eqref{eq:zero-atom}.

  Suppose that $\mu$ has an atom $a>0$ at $x>0$. Then, as $\eta\downarrow0$,
  $\eta s(x+i\eta)\to ia$, so $\Imc s\to\infty$ and $|s|/\Imc s$ remains
  bounded. On the other hand,
  \[
    \left|\E\frac{Rs}{1+cRs}\right|
    \le\frac{|s|}{c\Imc s}
  \]
  remains bounded, whereas $1+(x+i\eta)s(x+i\eta)$ diverges. This contradicts
  \eqref{eq:main-fixed-point1} and excludes positive atoms.
\end{proof}

\subsection{Support}

We follow the method of Silverstein and Choi~\cite{SilversteinChoi1995}.

\begin{lemma}
  \label{lem:negative-boundary}
  Let $R$ have a continuous density $f$ that is strictly positive on
  $(0,\infty)$. Fix $\sigma<0$ and put
  \[
    r_\sigma=-\frac1{c\sigma}>0.
  \]
  If $w=u+iy\to \sigma$ with $y>0$, then
  \begin{equation}
    \Imc\E\frac{R}{1+cRw}
    \longrightarrow
    -\frac{\pi r_\sigma f(r_\sigma)}{c|\sigma|}<0.
    \label{eq:negative-boundary}
  \end{equation}
\end{lemma}

\begin{proof}
  For $w=u+iy$,
  \[
    \Imc\E\frac{R}{1+cRw}
    =
    -cy\int_0^\infty
    \frac{r^2f(r)}
    {(1+cru)^2+(cry)^2}\,\dd r.
  \]
  For $u<0$ near $\sigma$, set
  \[
    r_u=-\frac1{cu}.
  \]
  The only singular point of the real denominator approaches $r_\sigma$.
  Localize the integral to a fixed small neighborhood of $r_\sigma$ and use
  \[
    1+cru=cu(r-r_u).
  \]
  After the change of variables $r=r_u+yt$, the localized integral converges, by
  the standard Poisson-kernel approximate-identity calculation, to
  \[
    -\frac1c
    \int_{\R}
    \frac{r_\sigma^2f(r_\sigma)}
    {\sigma^2t^2+r_\sigma^2}\,\dd t
    =
    -\frac{\pi r_\sigma f(r_\sigma)}{c|\sigma|}.
  \]
  Outside the chosen neighborhood of $r_\sigma$, $|1+cru|$ is uniformly bounded
  away from zero relative to $1+r$ for $u$ sufficiently close to $\sigma$. The
  corresponding integrand is dominated by an integrable multiple of $f(r)$ and
  converges pointwise to zero. This proves \eqref{eq:negative-boundary}.
\end{proof}

\begin{theorem}[Support]
  \label{thm:support}
  Assume $\lambda v>0$. If $c<1$, there is a unique $\sigma_*>0$ satisfying
  \begin{equation}
    c\sigma_*^2\E\frac{R^2}{(1+cR\sigma_*)^2}=1.
    \label{eq:sstar}
  \end{equation}
  With
  \[
    x_-:=\E\frac{R}{1+cR\sigma_*}-\frac1{\sigma_*},
  \]
  one has $x_->0$ and
  \[
    \supp\mu_{c,\lambda,v}=[x_-,\infty).
  \]
  If $c\ge1$, then
  \[
    \supp\mu_{c,\lambda,v}=[0,\infty).
  \]
  In particular, the support is unbounded whenever $\lambda v>0$.
\end{theorem}

\begin{proof}
  Let $f$ be the strictly positive continuous lognormal density and define, for
  $\sigma>0$,
  \[
    F(\sigma)=\E\frac{R}{1+cR\sigma}-\frac1\sigma,
    \qquad
    \xi_\sigma=\frac{cR\sigma}{1+cR\sigma},
    \qquad
    Q(\sigma)=\frac1c\E \xi_\sigma^2.
  \]
  Then
  \[
    F'(\sigma)=\frac{1-Q(\sigma)}{\sigma^2}.
  \]
  If $c<1$, $Q$ is continuous and strictly increasing from $0$ to $1/c>1$, which
  proves existence and uniqueness in \eqref{eq:sstar}. Moreover $F(0+)=-\infty$,
  $F(\infty)=0+$, and $F$ increases up to its unique maximum $x_-=F(\sigma_*)$
  and decreases thereafter. At $\sigma_*$, $\E \xi_{\sigma_*}^2=c$ and
  $0<\xi_{\sigma_*}<1$, so $\E \xi_{\sigma_*}>c$ and
  \[
    \sigma_*x_-=\frac{\E \xi_{\sigma_*}-c}{c}>0.
  \]

  Put
  \[
    \Psi(w)=\E\frac{R}{1+cRw}-\frac1w.
  \]
  This function is holomorphic on a complex neighborhood of the positive real
  axis, and its restriction there is $F$. For every $\sigma_0\in(0,\sigma_*)$,
  \[
    \Psi'(\sigma_0)=F'(\sigma_0)>0.
  \]
  Hence the holomorphic inverse theorem gives a local analytic inverse $\psi$ of
  $\Psi$ near $x_0=F(\sigma_0)$. Since
  \[
    \psi'(x_0)=\frac1{F'(\sigma_0)}>0,
  \]
  we have, for sufficiently small $\eta>0$,
  \[
    \Imc\psi(x_0+i\eta)>0.
  \]
  Thus this local inverse maps an upper-half-plane neighborhood into $\C_+$. By
  \cref{lem:uniqueness}, it must agree there with the Stieltjes transform $s$.

  Because $F$ maps $(0,\sigma_*)$ strictly increasingly from $-\infty$ onto
  $(-\infty,x_-)$, these local inverse branches patch together and show that $s$
  extends real-analytically across $(0,x_-)$. Stieltjes inversion therefore
  gives
  \[
    \mu((0,x_-))=0.
  \]

  It remains to exclude any gap above $x_-$. Suppose that an open interval
  $J\subset(x_-,\infty)$ has zero $\mu$-mass. Then $s$ extends holomorphically
  across $J$, and for every $x\in J$ the boundary value $s(x)$ is finite and
  real.

  We consider the three possible signs of $s(x)$.

  First, $s(x)>0$ is impossible. Indeed, the fixed-point equation extends to the
  gap and gives
  \[
    x=F(s(x)).
  \]
  Since $F(\sigma)\le x_-$ for every $\sigma>0$, this contradicts $x>x_-$.

  Second, $s(x)=0$ is impossible. Since
  \[
    s'(x)
    =
    \int\frac1{(t-x)^2}\,\mu(\dd t)>0,
  \]
  analyticity gives
  \[
    s(x+i\eta)
    =
    i\eta s'(x)+O(\eta^2).
  \]
  Thus $|s(x+i\eta)|/\Imc s(x+i\eta)$ remains bounded. Moreover,
  \[
    \left|
    \frac{Rs}{1+cRs}
    \right|
    \le
    \frac{|s|}{c\,\Imc s}.
  \]
  For every fixed $R$, the integrand tends to zero as $\eta\downarrow0$, so
  dominated convergence in
  \[
    1+zs=\E\frac{Rs}{1+cRs}
  \]
  would give $1=0$, a contradiction.

  Finally, suppose $s(x)=\sigma<0$. Since $s(x+i\eta)\to \sigma$ from $\C_+$,
  \cref{lem:negative-boundary} gives
  \[
    \Imc\E\frac{R}{1+cR\,s(x+i\eta)}
    \longrightarrow
    -\frac{\pi r_\sigma f(r_\sigma)}{c|\sigma|}<0,
    \qquad
    r_\sigma=-\frac1{c\sigma}.
  \]
  Since $\sigma\ne0$, we also have
  $\Imc(-1/s(x+i\eta))\to0$. Hence $\Imc\Psi(s(x+i\eta))$ has a strictly
  negative limit, contradicting $\Psi(s(x+i\eta))=x+i\eta$.

  Thus no open gap can lie above $x_-$. Since the complement of the support is
  open,
  \[
    \supp\mu_{c,\lambda,v}=[x_-,\infty).
  \]

  If $c\ge1$, then for every $\sigma>0$,
  \[
    \sigma F(\sigma)
    =
    \frac1c\E \xi_\sigma-1
    <
    \frac1c-1
    \le0.
  \]
  Hence a real boundary value $s(x)>0$ is impossible at every $x>0$. The zero
  and negative alternatives are excluded exactly as above. Therefore the
  limiting measure has no open gap in $(0,\infty)$.

  Since the covariance matrices are positive semidefinite,
  $\supp\mu_{c,\lambda,v}\subseteq[0,\infty)$. If $c>1$, the zero atom from
  \cref{prop:atoms} places $0$ in the support. If $c=1$ and $0$ were not in the
  support, some interval $(0,\delta)$ would be disjoint from the support,
  contradicting the preceding no-gap argument. Consequently
  \[
    \supp\mu_{c,\lambda,v}=[0,\infty).
  \]

  Unboundedness can also be seen from \eqref{eq:NC-moments}: the one-block
  partition gives
  \[
    M_k\ge c^{k-1}\E R^k
    =c^{k-1}e^{\lambda vk(k-1)/2},
  \]
  whose $k$th root diverges.
\end{proof}

The atom argument requires only $R>0$ almost surely. The support proof applies
whenever $R$ has a strictly positive continuous density on $(0,\infty)$, and
the moment formula holds for positive jump laws with moments of all orders.

When $\lambda v>0$, unbounded support and almost-sure ESD convergence imply
$\|S_n\|_{\op}\to\infty$ almost surely: for every fixed $M$, the limit assigns
positive mass to $(M,\infty)$, so eventually an eigenvalue lies there. No growth
rate follows from this argument.

\subsection{The unit-modulus case}

If $v=0$, then $X^2=1$ almost surely and $R\equiv1$. Equation
\eqref{eq:main-fixed-point2} becomes
\[
  \frac1s+z=\frac1{1+cs}.
\]
This is the MP equation, also obtained when $\lambda=0$.
For $v=0$ the identity $R_n\equiv1$ holds at every degree, and
\cref{prop:sharper-offdiag} recovers the larger MP range of
Yaskov~\cite{Yaskov2025}.

\begin{corollary}[The unit-modulus case]
  \label{cor:unit-modulus}
  Assume \eqref{eq:base-assumptions}, $v=0$, $p=\binom nd\to\infty$, $p/N\to
  c\in(0,\infty)$, and
  \begin{equation}
    \min(d,n-d)=o(n).
    \label{eq:unit-degree}
  \end{equation}
  Then $\ESD(S_n)$ converges almost surely to the ordinary MP law. Moreover,
  \eqref{eq:unit-degree} is necessary and sufficient for the uniform
  quadratic-form concentration
  \[
    \sup_{\|A\|_{\op}\le1}
    \E\left|\frac{x^{\mathsf T}Ax}{p}-\frac{\Tr A}{p}\right|^2\longrightarrow0.
  \]
\end{corollary}

\begin{proof}
  Since $v=0$, $X^2=1$ almost surely; together with $\E X=0$, this makes $X$
  Rademacher, so $\E X^3=0$ and $R_n\equiv1$. Set $k_n=\min(d,n-d)$. For every
  $n$ with $d>n/2$, complementing all coordinate sets identifies $xx^{\mathsf
  T}$, under a deterministic coordinate permutation, with the degree-$k_n$ outer
  product: the samplewise common factor $\prod_{i=1}^nX_i$ cancels. We may
  therefore work at degree $k_n=o(n)$ for the full sequence. The diagonal term
  is identically zero. Since $v=0$, the exponential factor in
  \eqref{eq:sharper-offdiag} equals one, and \cref{prop:sharper-offdiag} applies
  throughout $k_n/n\to0$. For all sufficiently large $n$, it gives
  \[
    \E\left|\frac{x^{\mathsf T}Ax}{p}-\frac{\Tr A}{p}\right|^2
    \le\frac{2k_n(n-k_n)}{n(n-1)}\|A\|_{\op}^2.
  \]
  Thus \eqref{eq:abstract-radial} holds with $\nu=\delta_1$ and
  \[
    \eps_n=\frac{2k_n(n-k_n)}{n(n-1)}.
  \]
  \Cref{thm:abstract-radial} gives the MP law in probability.
  Since $p=\binom{n}{k_n}\to\infty$, eventually $k_n\ge1$ and $p\ge n$.
  Thus $\sum_ne^{-ap}<\infty$ for every $a>0$, and \cref{lem:as-upgrade}
  gives almost-sure convergence.

  For sharpness, failure of \eqref{eq:unit-degree} gives a subsequence on which
  $\min(d,n-d)/n\ge\eps>0$; pass further so that $d/n\to\gamma\in[\eps,1-\eps]$.
  Fix $D=\{1,2\}$ and let $A$ be the symmetric partial-permutation matrix with
  $A_{I,I\triangle D}=1$ when $|I\cap D|=1$ and all other entries zero. Then
  $\|A\|_{\op}=1$, $\Tr A=0$, and
  \[
    \frac{x^{\mathsf T}Ax}{p}
    =\frac{2\binom{n-2}{d-1}}{\binom nd}X_1X_2
    =\frac{2d(n-d)}{n(n-1)}X_1X_2.
  \]
  Its $L^2$ norm tends to $2\gamma(1-\gamma)>0$. This proves necessity.
\end{proof}

\section{Two questions}
\label{sec:remarks}

For $v>0$ and $d^2/n\to\infty$,
\[
  \E R_n^2=\E(1+v)^K\ge(1+v)^{d^2/n}\longrightarrow\infty.
\]
The uniform second-moment bound is lost, but a weak limit with loss of moments
is not excluded. The prefactor $d/n$ may still compensate for the exponentials
in \eqref{eq:delta-kappa}. The present proof, however, establishes the $L^2$
comparison with $P_n$ only for bounded $d^2/n$. What are the radius and spectral
limits beyond that range, and what normalization, if any, is needed?

When $\E X^3=0$, \cref{prop:sharper-offdiag} has prefactor
$\rho_1=2d(n-d)/(n(n-1))$. The general bound replaces it by $4\rho_1$ and
increases the exponential factor. Both losses come from Cauchy--Schwarz over
$D\subseteq S_1$ in the chaos expansion. Does the sharper bound hold without the
third-moment assumption?

\section*{Acknowledgements}
Language-model tools were used as writing and editing assistants during the
preparation of the manuscript. The author is responsible for its content.

\end{document}